\documentclass[a4paper,12pt]{amsart}

\usepackage{epsfig}
\usepackage{amsthm,amsfonts}
\usepackage{amssymb,graphicx,color}
\usepackage[all]{xy}
\usepackage{verbatim}
\usepackage{hyperref}
\usepackage{enumitem}

\usepackage[a4paper,top=2.5cm,bottom=2.5cm,left=3cm,right=3cm]{geometry}

\newtheorem{theorem}{Theorem}[section]
\newtheorem*{theorem*}{Theorem}
\newtheorem{lemma}[theorem]{Lemma}
\newtheorem{corollary}[theorem]{Corollary}

\newtheorem{definition}[theorem]{Definition}

\newcommand{\R}{\mathbb{R}}

\newcommand{\C}{\mathbb{C}}

\title[On real polynomial local homeomorphisms ]{On  real polynomial local homeomorphisms of three dimensional space}

\author{Alexandre Fernandes}
\address{Alexandre Fernandes -
Departamento de Matem\'atica, Universidade Federal do Cear\'a, Av.
Humberto Monte, s/n Campus do Pici,  60440-900, Brasil.}
\email{alex@mat.ufc.br}

\author{Zbigniew Jelonek}
\address{Zbigniew Jelonek -
Instytut Matematyczny, Polska Akademia Nauk, \'Sniadeckich 8, 00-656 Warszawa, Poland.}
\email{najelone@cyf-kr.edu.pl}

\thanks{Z. Jelonek is partially supported by the grant of Narodowe Centrum Nauki, grant number 2015/17/B/ST1/02637. A. Fernandes is partially supported by CNPq grant  N 302764/2014-7 }

\begin{document}


\begin{abstract}
We prove that the set of non-properness of  a polynomial mapping of the three dimensional space which is a local homeomorphism cannot be homeomorphic to the real line $\R.$
\end{abstract}
\maketitle

 \section{Introduction}

\medskip

A local homeomorphism $F$ of an affine space  is a global homeomorphism if and only if it is a proper mapping in the topological sense, i.e. $F^{-1}(K)$ is compact for any compact subset $K\subset\R^n$. Let us denote  the set of values where $F$ is not proper by $S_F$
(see \cite{Je1} and \cite{Je2}), i.e.,
$$S_F=\{ y\in\R^n \ | \exists \ \mbox{ unbounded sequence } (x_k) \mbox{ in } \R^n; \ \mbox{with} \ F(x_k)\to y \}.$$ 
A local homeomorphism $F\colon\R^n\rightarrow\R^n$ is a global homeomorphism  if and only if the set $S_F$ is empty. The set $S_F$ plays an important role  in the study of the topological behavior of  local homeomorphism $\R^n\rightarrow\R^n.$ In \cite{Je1} and \cite{Je2}, the second named author investigated the geometry of the set $S_F$ for polynomial mappings $\C^n\rightarrow\C^n$ and
$\R^n\to\R^n.$ In particular we have:

\begin{theorem}(see \cite{Je2})
Let  $f:\R^n\rightarrow\R^n$ be a  polynomial mapping whose Jacobian
 nowhere vanishes. If ${\rm codim} \ S_f\geq 3$ then $f$ is
a diffeomorphism 
(and consequently $S_f=\emptyset$).
\end{theorem}

\vspace{5mm}

On the other hand the example of Pinchuk   shows
that there are  polynomial mappings $f:\R^n\to\R^n$ whose Jacobian 
nowhere vanishes and with codim $S_f=1.$ In  Pinchuk's original example we have $n=2$ and  the set $S_f$ is homeomorphic to the real line $\R$
(see  \cite{cam}, \cite{gw}). In the original Pinchuk example the mapping $f$ is not surjective, in fact it omits two points - $(0,0)$ and $(-1,0)$ -see  \cite{gw}. However we can easily modify the Pinchuk example in this way that a modified mapping is a surjection. Indeed, let 
$\phi : \C\ni z \mapsto z^3/3+z^2/2\in \C$, where we treat $\C$ as $\R^2.$ The mapping $\phi$ is singular only for $z=0$ and $z=-1$, hence the composition $F=\phi\circ f$ has not singular points. The mapping $F$ can omit only points $\phi(0)=0$ and $\phi(-1)=1/6.$ However $\phi(0)=\phi(-3/2)$ and 
$\phi(-1)=1/6=\phi(1/2).$ Hence the mapping $F$ is a surjection. This means that also surjective counterexamples of Pinchuk type exists. Note that the set $S_F$ is still an "irreducible" connected curve.

However, the Pinchuk  example can be easily generalized for every $n>1$ (and then  $S_f$ is homeomorphic to $\R^{n-1}$ or $f$ is surjective mapping and $S_f$ is "irreducible" connected hypersurface).
Hence the only interesting case
is that of codim $S_f=2$, and in \cite{Je2} the following was stated:

\vspace{5mm}

{\bf Real  Jacobian  Conjecture.}  {\it  Let  $f:\Bbb
R^n\rightarrow\Bbb R^n$ be a polynomial mapping whose Jacobian
nowhere vanishes. If ${\rm codim} \ S_f\geq 2$ then $f$ is a diffeomorphism
(and consequently $S_f=\emptyset$).}

\vspace{5mm}

The Real Jacobian Conjecture (in dimension $2n$)
implies the famous
Jacobian Conjecture (in dimension $n$) - see \cite{Je2}, \cite{arno}:

\vspace{5mm}

{\bf Jacobian  Conjecture.}  {\it  Let  $f:\Bbb
C^n\rightarrow\Bbb C^n$ be a polynomial mapping whose Jacobian
 nowhere vanishes. Then $f$ is an isomorphism.}

\vspace{5mm}
\noindent Moreover, the Real Jacobian Conjecture is true in dimension two (see \cite{FMS}, \cite{Je2}). In dimension three to prove the Real Jacobian Conjecture it is enough to assume that the set $S_f$ is a curve. In this case the curve $S_f$ is automatically the union of real parametric curves, i.e., $S_f=\bigcup^r_{i=1} S_i$, where $S_i=\phi_i(\R)$ and $\phi_i: \R\to\R^3, \ i=1,...,r$ are  non-constant polynomial mappings (see 
\cite{Je2}). 

The aim of this note is to show that the set $S_F$ of non-properness of a polynomial mapping $f:\R^3\to\R^3$  which is a local homeomorphism cannot be homeomorphic to  $\R.$ Moreover, we show that if additionally $f$ is a surjection, then $S_f$ has to have at least three ends.

\section{Main results}

\begin{definition}
{\rm Let  $X\subset \R^3$ be a semi-algebraic curve. Denote by $B_R$ the open ball with  center  $0$ and  radius $R.$  We say that $X$ has $n$ {\it ends at infinity} if the set $X\setminus B_R$ has $n$ connected components for every sufficiently large $R>0$. } 
\end{definition}

\begin{definition}
{\rm A  subset of $\R^3$ is called a \textit{line} if it is closed and homeomorphic to $\R$.}
\end{definition}

It is easy to see that a line has exactly two ends at infinity.

\begin{theorem}\label{main-theorem} If $F:\R^3\to\R^3$ is a polynomial local diffeomorphism,
then $S_F$ cannot be a line.
\end{theorem}
\begin{proof} 
Let $S_R\subset\R^3$ be a sphere centered at the origin and with a sufficiently large radius $R>0$ such that any other spere centered at the origin and with radius larger than or equal to  $R$ is transversal to $S_F$. Thus, $Y_R:=\R^3\setminus (S_F\cup B_R)$ has the same homotopy type of $\R^3\setminus L$ where $L$ is an affine line, hence the fundamental group of $\R^3\setminus (S_F\cup B_R)$ is cyclic and infinite. Take $X_R:=\R^3\setminus F^{-1}(B_R\cup S_F)=F^{-1}(Y_R).$ Hence if the set $X_R$ is connected, then the fundamental group of $X_R$ is also cyclic and infinite.

Take $E=F^{-1}(S_F).$ 
Note that $S_F\neq\emptyset$ implies that $E\neq\emptyset$. Indeed, if not, the fundamental group of $\R^3\setminus E=\R^3$ is trivial. Moreover  $\pi_1(\R^3\setminus S_F)$ is an infinite  group. By the covering property,  the topological degree of $F$ is equal to the  index $[\pi_1(\R^3\setminus S_F):F_*(\pi_1(\R^3))]=[\pi_1(\R^3\setminus S_F):\{1\}]=\infty.$ This is a contradiction, because  the mapping $F$ is polynomial
and it has finite fibers.
Now we need the following:

\begin{lemma}\label{lem}
Let $K$ be the Zariski closure of $S_F.$ Then $K=S_F\cup P$, where $P$ is a finite set of points. 
\end{lemma}

\begin{proof}
Indeed, by \cite{Je2}, $S_F$ is a parametric curve, i.e., there are non-constant polynomials  $\psi_1(t), \psi_2(t),\psi_3(t)$ such that $S_F=(\psi_1, \psi_2,\psi_3)(\R):=\psi(\R).$ Consider the
field $L=\R(\psi_1,\psi_2,\psi_3)$. By the L\"uroth Theorem
there exists a rational function $g(t)\in \R(t)$ such that
$L=\R(g(t)).$ In particular there exist  $f_1, f_2, f_3\in \R(t)$
such that $\psi_i(t)=f_i(g(t))$ for $i=1,2,3.$ In fact, we have
two induced maps $\overline{f} : \Bbb P^1(\R)\rightarrow \overline{X}\subset \Bbb
P^m(\R)$ and $\overline{g}: \Bbb P^1(\R)\rightarrow\Bbb P^1(\R)$. Here $\overline{X}$ denotes  the Zariski projective
closure of $S_F.$
Moreover, $\overline{f}\circ \overline{g}=\overline{\psi}.$ Let
$A_\infty$ denote the unique point at infinity 
of $\overline{X}$ and let
$\infty=\overline{f}^{-1}(A_\infty).$ Then
$\overline{g}^{-1}(\infty)=\infty$,  i.e., $g\in \R[t].$
Similarly $f_i\in \R[t].$ Hence $\psi=f\circ g$, where $f:\R\to X$ is a birational and polynomial mapping and $g:\R\to \R$ is a polynomial mapping.

Since $S_F$ has two ends, we see that the mapping $g$ is surjective. Hence we can assume that $\psi=f.$
We have $K=f(\C)\cap \R^3$ and it is enough to prove that there is only a finite number of points in the set
$f(\C\setminus \R)\cap \R^3.$ But if $z=f(x)\in \R^3$ and $x$ is not a real number, then also $z=f(\overline{x})$,
where the bar denotes the complex conjugation. In particular the fiber of $f$ over $z$ has at least two points.
Since the mapping $f$ is birational, the number of such points $z$ is finite.
\end{proof}

Let $E'$ be the Zariski closure of $E.$
By Lemma \ref{lem} we have  $E'=E\cup F^{-1}(P)$ and $F^{-1}(P)$ is a finite set of points. Hence at infinity $E$ looks like an algebraic variety. Let $\mathcal S$ be the one-point compactification of $\R^3.$
If $E$ has $k$ branches at infinity, then its compactification $\tilde{E}\subset {\mathcal S}$ is homeomorphic to the bouquet of $k$ circles. Using the Alexander duality we can compute the cohomology of ${\mathcal S}\setminus \tilde{E}=\R^3\setminus E.$ 
We have $$\tilde{H}_i(\tilde{E})\cong \tilde{H}^{3-i-1}(\R^3\setminus E).$$
Hence $\tilde{H}^p(\R^3\setminus E)=0$ for $p\not=1$ and  $\tilde{H}^1(\R^3\setminus E)={\Bbb Z}^k.$
In particular
$\chi(\R^3\setminus E)=1-k.$ In the same way we have $\chi(\R^3\setminus S_F)=0.$ Note that $\R^3\setminus E$ is a finite topological  covering of $\R^3\setminus S_F$ of degree say $d>0.$ Thus $\chi(\R^3\setminus E)= d \chi(\R^3\setminus S_F)=0.$ 
Consequently, $k=1$ and $E$ has only one  branch at infinity. In particular  $F(E)= S_F$.

Now we show that the set  $X_R$ is connected. In fact we will prove the following statement:

\vspace{5mm}

$(*)$ {\it Let $F:\R^3\to\R^3$ be a polynomial mapping as above. For every $R \ge 0$ the set $X_R:=F^{-1}(\R^3\setminus B_R)\setminus E$ is connected.}

\vspace{5mm}
Note that it is enough to prove that  $X:=F^{-1}(\R^3\setminus B_R)$ is connected. We prove this in several steps. Let $O\in \R^3\setminus S_F$ and let $L$ be a half-line with  origin  $O$, such that $L\cap S_F=\emptyset.$
Hence we can lift $L$ in a  $d=\deg F$ different ways; let $L_1,\ldots ,L_d$ be all these lifts. We can assume that $O$ is the center of coordinates.
Let $S_R$ be the sphere $\{ x: ||x||=R\}.$ By $S^i_R$ we will denote the connected component of  $F^{-1}(S_R)$ which intersects the curve $L_i$. Of course it is possible that $S_i=S_j$ for $i\not =j.$ Take $\phi_R(x):=||F(x)||-R.$ For $R>0$ the hypersurface $S_i$ divides $\R^3$ into two open connected components (see \cite{Lim}). Take a point $x\in S_i$ and consider a small  ball $B$ around $x$. The function $\phi_R$  has
signs $+$ and $-$ in the set $B\setminus S_i.$  A component  $\R^3\setminus S_i$, which contains this part of the ball $B$, on which 
$\phi_R|_B$ is negative will be called the interior of $S_i$ and denoted  by $Int(S_i)$; the another component is denoted by $Ext(S_i).$

\begin{lemma}\label{1}
There is an $\eta>0$ such that if $x\in S^i_R$ and $B(x,\eta)\cap S^j_{R'}$ $ \not =\emptyset$, then $S^j_R=S^i_R.$
\end{lemma} 

\begin{proof}
Let $A_i=L_i\cap S^i_R$ and let $\alpha(t), t\in [0,1],$ be a curve in $S^i_R\setminus E$ which joins $A_i$ and $x.$ For every point 
$x_t:=\alpha(t)$ let $V_t$ be a small ball with center in $x_t$ such that 

1) $F|_{V_t}: V_t\to F(V_t)$ is a homeomorphism,

2) $V_t \cap E=\emptyset.$

Let $\beta=F\circ \alpha.$ The curve $\beta$ is covered by the sets $U_t=F(V_t).$ Let $\epsilon$ be the Lebesgue number of this covering. 
Define the curve $\beta'(t)=\overline{O\beta(t)}\cap S_{R'}.$ If $|R-R'|<\epsilon$, then the curve $\beta'$ is contained in 
$\bigcup_{t\in [0,1]} U_t.$ In particular if we lift  $\beta'$ so that it starts at  $A_i':=L_i\cap S^i_{R'}$, then we obtain a 
curve $\alpha'$ which has it end in  $V_{1}.$ In particular we can join  $\alpha'(1)$ to $x.$ This implies that
if we take $\eta<\epsilon$ then the conclusion of Lemma \ref{1} is true. 
\end{proof}

\begin{lemma}\label{2}
The set $X_R$ is connected if and only if for every $i,j=1,\ldots,d$ the set $S^j_R$ is not contained in the interior of $S^i_R.$ 
\end{lemma}

\begin{proof}
Assume that for every $i,j=1,\ldots,d$ the set $S^j_R$ is not contained in the interior of $S^i_R.$ Take points $a,b\in X_R.$ Of course there is a path $\alpha$  in $\R^3\setminus E$ such that $\alpha(0)=a$ and $\alpha(1)=b.$ If $\alpha$ has no common points with $S^i_R, i=1,\ldots,d$, then $\alpha$ is contained in $X_R$ and we are done. Assume that $\alpha$ has a common point with $S^i_R.$ Note that by  assumption the point $\alpha(1)$ is not  in the interior of $S^i_R$. Let $t_0$ be the first point in $S^i_R$ and $t_1$ be the last point in $S^i_R$; by  assumption we have $t_1<1.$ Modify the curve $\alpha$ replacing the path $\alpha|_{[t_0,t_1]}$ by a curve $\gamma$ which is contained in $S^i_R$ and joins  $\alpha(t_0)$ and $\alpha(t_1).$ By a slight further modification of $\alpha$ as in the proof of Lemma \ref{1} above, we can assume that
$\alpha|_{[t_0,t_1]}$ lies in $X_R.$ Continuing,  we obtain a curve $\alpha$  completely contained in $X_R.$   

\vspace{5mm}

If the interior of $S^i_R$ contains $S^j_R$, then the exterior of $S^j_R$ is also contained in $Int(S^i_R)$ and points from the exterior of $S^i_R$ and the exterior of $S^j_R$ cannot be connected. 
\end{proof}

\begin{lemma}\label{3}
Assume that the set $X_R$ is connected. Then for $R'>R$ sufficiently close to $R$ the set $X_{R'}$ is also connected.
\end{lemma}

\begin{proof}
Take a (given) point $a\not\in B_R\cup S_F.$ Since the mapping $F|_{X_R}: X_R\to Y_R$ is a topological covering, the set $X_R$ is connected if and only if all points from the fiber $F^{-1}(a)=\{ a_1,\ldots, a_d\}$ can be connected to each other. If $X_R$ is connected then there is a curve $\alpha\subset X_R$ which connects all points $a_i, i=1,\ldots,d.$ Let $c=\mbox{inf}_{t\in [0,1]}  ||F(\alpha(t))||.$ Then $c>R.$ Consequently our statement is true for every $R'$ with  $c>R'>R.$
\end{proof}

Now we can prove  statement $(*).$ Note that for small $R$ the set $X_R$ is connected. Let $\mathcal A= \{ R\in \R: \ X_R \ \mbox{is not  connected}\}. $ Assume that  $\mathcal A$ is non-empty and take $a=\mbox{inf} \{ R: R\in \mathcal{A}\}.$ By Lemma \ref{3} the set $X_a$ is not connected. By Lemma \ref{2} there is an ${S^i}_a$ whose interior contains some $S^{j_1}_a,\ldots, S^{j_k}_a.$ In particular there is a path $\alpha(t), t\in [0,1],$ such that $||F(\alpha ((0,1))||\subset [0,a)$ and $\alpha (0)\in {S^i}_a, \alpha (1)\in {S^j}_a$, where $j\in \{j_1,\ldots, j_k\}.$ We can assume that $\alpha$ is transversal to $S^i_a$ and $S^j_a.$ Hence the function $q(t):=||F(\alpha(t))||, t\in [0,1],$ decreases  near $0$
and increases near $1.$ In particular there is an $\epsilon>0$ such that $q$ decreases on $[0,\epsilon]$ and increases on $[1-\epsilon, 1].$
Let $\beta =\mbox{max}_{t\in [\epsilon, 1-\epsilon]} q(t)$ and take $R_0=\mbox{max} \{ q(\epsilon), q(1-\epsilon), \beta\}.$
For $R_0 < R<a$  the function $q(t)-R$ has only two zeroes, one near $0$ and the other near $1.$ 

Take $ R_0\le R<a$ and let $q(t_0)=R, q(t_1)=R$, where $t_0$ is close to $0$ and $t_1$ is close to $1.$  
By Lemma \ref{1} the set $S^k_R$ which contains   $\alpha(t_0)$ is equal to $S^i_R,$ and the set $S^l_R$ which contains   $\alpha(t_1)$ is equal to $S^j_R$. Since $R<a$ and on the curve $\alpha(t), t\in (t_0,t_1),$ we have $q(t)<R,$ we conclude that $S^i_R=S^j_R$.
Consequently, $S^i_a=S^j_a$, which is a contradiction.

Hence the set $\mathcal A$ is empty and statement $(*)$ is proved.

\vspace{5mm}

Since $F(E)=S_F$, there is a point $x\in E\setminus F^{-1}(B_R)$ such that $F(x)\in S_F\setminus B_R.$ 
Let $V\subset \R^3\setminus F^{-1}(B_R)$ be a neighbourhood of $x$ such that $F|_V:V\to W$ is a diffeomorphism onto an open neighbourhod $W\subset \R^3\setminus B_R$ of $y$. We can choose a generator of $\pi_1(Y_R,y)$ to be the class of a loop $\gamma$ inside any tubular neighbourhood of $S_F$ in $W.$ So, it is possible to choose a generator $[\gamma]$ of $\pi_1(Y_R,y)$ inside  $W.$ Take  $\tilde\gamma=F|_W^{-1}(\gamma).$ Hence $[\tilde{\gamma}]\in \pi_1(X,x)$ and $F_*([\tilde{\gamma}])=[\gamma].$ This implies that $[\pi_1(Y_R,y):F_*(\pi_1(X_R,x))]=1$, so the general fiber of  $F|_{X_R}$ is one point. This means that the topological degree of the covering $F: \R^3\setminus E\to \R^3\setminus S_F$ is one. Consequently, $F$ is injective. Now by the Bia\l ynicki-Rosenlicht Theorem  (see \cite{BB})  $F$  is a bijection. 
\end{proof}

From the proof of our result we have:

\begin{corollary}
If $F:\R^3\to\R^3$ is a polynomial local diffeomorphism and ${\rm dim}\ S_F=1$, 
then $S_F$ cannot have only a one end.
\end{corollary}

\begin{proof}
Assume that $S_F$ has exactly one end. If $B_R$ is sufficiently large ball, then $\R^3\setminus (B_R\cup S_F)$ is simply connected and we can finish as above.
\end{proof}

Note that the main part of the proof of Theorem \ref{main-theorem} works also if $S_F$ is an arbitrary  curve with two ends (which possibly it is not homeomorphic to $\R$). The assumption that $S_F$ is homeomorphic to $\R$ we need only to know that $F(F^{-1}(S_F))=S_F.$ Hence for surjective mapping $F$ we can strenghten our result to the case when $S_F$ has two ends. In particular we have:

\begin{corollary}
Let  $F:\R^3\to\R^3$ be a polynomial local diffeomorphism and ${\rm dim}\ S_F=1$. If $F$ is surjective,
then $S_F$  has at least three ends.
\end{corollary}

\end{document}